\documentclass[pdflatex,sn-mathphys-num]{sn-jnl}

\usepackage{graphicx}%
\usepackage{multirow}%
\usepackage{amsmath,amssymb,amsfonts}%
\usepackage{amsthm}%
\usepackage{mathrsfs}%
\usepackage[title]{appendix}%
\usepackage{xcolor}%
\usepackage{textcomp}%
\usepackage{manyfoot}%
\usepackage{booktabs}%
\usepackage{algorithm}%
\usepackage{algorithmicx}%
\usepackage{algpseudocode}%
\usepackage{listings}%
\usepackage[pagewise]{lineno}
\usepackage{enumitem}

\theoremstyle{thmstyleone}%
\newtheorem{theorem}{Theorem}
\newtheorem{lemma}{Lemma}
\newtheorem{cor}{Corollary}
\theoremstyle{thmstyletwo}%
\newtheorem{example}{Example}%

\theoremstyle{thmstylethree}%
\newtheorem{definition}{Definition}%
\newcommand{\Fcal}{\mathcal{F}} 
 
\begin{document}

\title[Interpolative Metrics Are not new]
      {Interpolative Metrics Are Not New: A Study of Generalized Contractions in \( b \)-Suprametric  Spaces}


\author*[1]{\fnm{Hassan} \sur{Khandani}}\email{hassan.khandani@iau.ac.ir, khandani.hassan@gmail.com}



\affil*[1]{\orgdiv{Department of mathematics, Mah.C }, \orgname{Islamic Azad University}, \orgaddress{\street{University street}, \city{Mahabad}, \postcode{5915873549}, \state{West Azerbaijan province}, \country{Iran}}}




\abstract{Researchers recently introduced interpolative metric spaces and established fixed point theorems in this setting. We demonstrate that these metrics are a special case of \( b \)-metrics. On the other hand, suprametrics and \( b \)-suprametrics have also been introduced, and we show that \( b \)-suprametric spaces generalize \( b \)-metric spaces. We establish corresponding results previously presented in interpolative metric spaces in the framework of \( b \)-suprametric spaces with weaker conditions.}

\keywords{b-suprametric space, \(b\)-metric,  \(\Fcal\)-metric, Contraction-type mapping, Matkowski's fixed point theorem.}


\pacs[MSC Classification]{47H09; 47H10}

\maketitle

\section{Introduction}\label{sec1}
The Banach contraction principle is a cornerstone of fixed point theory, with numerous extensions and applications in nonlinear analysis and applied mathematics. While classical metric spaces rely on the strict triangle inequality, many real-world applications involve relaxed notions of distance, where this inequality is weakened. These generalized metrics, such as \( b \)-metrics and suprametrics, have become central to modern research. Establishing versions of the Banach contraction principle in these settings remains an active area of investigation. In this study, we present fixed point results for generalized contractions within these metric frameworks.\\
The concept of \( b \)-metrics was introduced by Bakhtin \cite{bakhtin1989contraction}, who defined quasi-metrics and extended the Banach fixed point theorem to these spaces. Later, Czerwik \cite{czerwik1993contraction} formalized the notion of \( b \)-metrics and extended Matkowski's fixed point theorem to \( b \)-metric spaces (see also \cite{kajanto2018note}). For further details on \( b \)-metric spaces, we refer the reader to \cite{karapinar2018short, berinde2022early}.\\
Berzig \cite{berzig2022first} introduced suprametrics, defined by the inequality \eqref{defbsupra} in Definition \ref{metrics} with \( s = 1 \), and demonstrated their applications to nonlinear matrix and integral equations. Building on this work, Berzig \cite{berzig2024nonlinear} defined \( b \)-suprametrics, a novel framework that generalizes and unifies both \( b \)-metrics and suprametrics. Furthermore, Berzig extended Matkowski's fixed point theorem to \( b \)-suprametric spaces (see Theorem \ref{berzic-matko}). For applications of suprametrics in various contexts, we refer the reader to \cite{panda2023extended, panda2023numerical}.

While \( b \)-metrics share some properties with classical metrics, such as non-negativity and symmetry, they differ significantly due to the relaxed triangle inequality. For instance, as shown in \cite{lu2019fundamental}, there exist \( b \)-metrics that are not continuous. This behavior is also observed \( b \)-suprametrics. Similarly, Example 2.10 in \cite{berzig2024nonlinear} provides a \( b \)-suprametric that is not continuous in each variable. In contrast, a suprametric is always continuous in each variable \cite{berzig2022first}. To address this limitation, Kirk and Shahzad \cite{kirk2014fixed} introduced the concept of \textbf{strong \( b \)-metrics}, which are continuous. Building on this idea, Berzig \cite{berzig2024strong} introduced \textbf{strong \( b \)-suprametrics} and showed that they are also continuous. Due to these variations, it becomes essential to define continuity explicitly for these generalized metrics. Cobzas et al. \cite{cobzas2020completion} introduced the concept of continuity for \( b \)-metrics, as formalized in Definition \ref{continuity}. This definition will be the foundation for all other metrics discussed throughout this manuscript.
\subsection{Preliminary Definitions and Results}
Before proceeding further, we first review some preliminary concepts and results that are central to this manuscript, as outlined below:

\begin{definition}\label{metrics}
The function \( \Delta: X \times X \to \mathbb{R}^{+} \) is called a semi-metric if it satisfies the following two conditions:
\begin{enumerate}
    \item \( \Delta(x, y) \geq 0 \) for all \( x, y \in X \), and \( \Delta(x, y) = 0 \) if and only if \( x = y \),
    \item \( \Delta(x, y) = \Delta(y, x) \) for all \( x, y \in X \).
\end{enumerate}
Furthermore, \( \Delta\) is called:
\begin{enumerate}[resume]
    \item \textbf{a \( b \)-metric} if there exists a constant \( s \geq 1 \) such that for all \( x, y, z \in X \),
    \[
    \Delta(x, y) \leq s \left( \Delta(x, z) + \Delta(z, y) \right).
    \]
    In this case, \( (X, \Delta) \) is called a \( b \)-metric space, and \( s \) is called the index of the \( b \)-metric space.
    \item \textbf{a strong \( b \)-metric} if there exists a constant \( s \geq 1 \) such that for all \( x, y, z \in X \),
    \[
    \Delta(x, y) \leq \Delta(x, z) + s \Delta(z, y).
    \]
    In this case, \( (X, \Delta) \) is called a strong \( b \)-metric space, and \( s \) is called the index of the strong \( b \)-metric space.
    \item \textbf{a \( b \)-suprametric} if there exist constants \( s \geq 1 \), \( c \geq 0 \) such that for all \( x, y, z \in X \),
    \[
    \Delta(x, y) \leq s \big( \Delta(x, z) + \Delta(z, y) \big) + c \, \Delta(x, z) \, \Delta(z, y).
    \]
    In this case, \( (X, \Delta) \) is called a \( b \)-suprametric space. If \( s = 1 \), then \( d \) is called a suprametric.
    \item \textbf{a strong \( b \)-suprametric} if there exist constants \( s \geq 1 \), \( c \geq 0 \) such that for all \( x, y, z \in X \),
    \begin{equation}\label{defbsupra}
    \Delta(x, y) \leq s \Delta(x, z) + \Delta(z, y) + c \, \Delta(x, z) \, \Delta(z, y).
    \end{equation}
    In this case, \( (X, \Delta) \) is called a strong \( b \)-suprametric space.
    \item \textbf{an interpolative metric} if there exist constants \( \alpha \in (0, 1) \) and \( c \geq 0 \) such that for all \( x, y, z \in X \),
    \[
    \Delta(x, y) \leq \Delta(x, z) + \Delta(z, y) + c \left( \Delta(x, z)^\alpha \cdot \Delta(z, y)^{1-\alpha} \right).
    \]
    In this case, \( (X, \Delta) \) is called an interpolative metric space.
\end{enumerate}
A sequence \( \{x_n\} \) in \( X \) converges to a point \( a \in X \) if \( \lim_{n \to \infty} \Delta(x_n, a) = 0 \). A sequence \( \{x_n\} \) in \( X \) is called a Cauchy sequence if for every \( \epsilon > 0 \), there exists \( m_0 \in \mathbb{N} \) such that \( \Delta(x_n, x_m) < \epsilon \) for all \( n, m > m_0 \). The space \( (X, \Delta) \) is called complete if every Cauchy sequence in \( X \) converges to a point in \( X \). A subset $Y\subset X$ is called bounded if there exists $M>0$ such that $\Delta(x,y)\leq M$ for all $x,y\in Y$.
\end{definition}
\begin{definition}[Picard Sequence and Bounded Orbits]
Let $(X,\Delta)$ be a $b$-metric space and $T: X\to X$ a mapping. For any $x \in X$, the \emph{Picard sequence} (or \emph{orbit}) of $T$ based at $x$ is the sequence $\{T^n x\}_{n=0}^\infty$ defined by:
\[
T^{n+1}x = T(T^n x), \quad \text{where} \quad T^0 x = x.
\]
We denote this sequence by $\mathcal{O}(x)$. 

The orbit $\mathcal{O}(x)$ is called \emph{bounded} if it is a bounded subset of $X$ in the sense of Definition \ref{metrics}. 
\end{definition}
\begin{definition}\label{continuity}[Cobzas et al., \cite{cobzas2020completion}]
The metric \( d \) is called \textbf{separately continuous} if
\[
\Delta(y_n, y) \to 0 \implies \Delta(x, y_n) \to \Delta(x, y);
\]
\( d \) is called \textbf{jointly continuous} if
\[
\Delta(x_n, x) \to 0 \text{ and } \Delta(y_n, y) \to 0 \implies \Delta(x_n, y_n) \to \Delta(x, y);
\]
for all sequences \( (x_n), (y_n) \in X \) and all \( x, y \in X \).
\end{definition}

\begin{definition}
Let \( \theta : \mathbb{R}^{+} \to \mathbb{R}^{+} \) be a non-decreasing mapping satisfying one of the following conditions:
\begin{itemize}
    \item[\((\Theta_1)\)] \( \theta^{n}(t) \to 0 \) as \( n \to \infty \) for each \( t > 0 \).
    \item[\((\Theta_2)\)] \( \sum_{n=0}^{\infty} \theta^{n}(t) < +\infty \) for each \( t > 0 \).
\end{itemize}
The set of all non-decreasing mappings satisfying \( (\Theta_1) \) and \( (\Theta_2) \) is denoted by \( \Theta_1 \) and \( \Theta_2 \), respectively.
\end{definition}

It is well known that if \( \theta \) is non-decreasing and \( (\Theta_1) \) holds, then \( \theta(t) < t \) for each \( t > 0 \) (see Lemma 2.1 in \cite{samet2012fixed}). In honor of Matkowski, the class of functions denoted by \( \Theta_1 \) is referred to as the set of Matkowski's functions \cite{matkowski1975integrable}. We also know that \( \Theta_2 \subset \Theta_1 \).

\begin{theorem}\label{berzic-matko}[Berzig, \cite{berzig2024nonlinear}]
Let \( (X, \Delta) \) be a complete \( b \)-suprametric space and \( T: X \to X \) be a mapping satisfying the following contraction condition:
\[
\Delta(Tx, Ty) \leq \theta(\Delta(x, y)) \text{ for all } x, y \in X,
\]
where \( \theta \in \Theta_1 \). Then \( T \) has a unique fixed point.
\end{theorem}

Karapinar et al. established the following Ćirić-type fixed point result in the setting of interpolative metric spaces:
\begin{theorem}[Theorem 2.2, \cite{karapinar2025some}]
Let \( 0 \leq \alpha < 1 \), \( c \geq 0 \), and \( (X, \Delta) \) be an \( (\alpha, c) \)-interpolative complete metric space. Let \( T: X \to X \) be a mapping. Suppose that there exists \( \theta \in \Theta_2 \) such that
\[
\Delta(Tx, Ty) \leq \theta(M(x, y)),
\]
for all \( x, y \in X \), where
\[
M(x, y) = \max\{\Delta(x, y), \Delta(x, Tx), \Delta(y, Ty)\}.
\]
Then, \( T \) has a unique fixed point in \( X \).
\end{theorem}

In this manuscript, we first demonstrate that interpolative metric spaces, recently introduced \cite{karapinara2024interpolative}, are not genuinely new metric spaces. Specifically, we prove that every interpolative metric is, in fact, a \( b \)-metric \cite{karapinara2024interpolative}. On the other hand, Berzig introduced \( b \)-suprametric spaces, which generalize \( b \)-metric spaces \cite{berzig2024nonlinear}. However, no example was provided to show that this generalization is proper (i.e., that \( b \)-suprametric spaces strictly contain \( b \)-metric spaces) \cite{berzig2024nonlinear}. We address this gap by constructing such an example and present our main results in the setting of \( b \)-suprametric spaces, which are more general than \( b \)-metric spaces.

Next, we extend the main result of Karapinar et al. \cite{karapinar2025some} to the setting of \( b \)-suprametric spaces under weaker conditions. As a key contribution, we directly derive a generalized version of Ćirić-type Matkowski fixed-point theorem in the context of \( b \)-metric spaces.

This manuscript is organized as follows. In Section \ref{mr}, we present an example to assert that suprametric and \( b \)-suprametric spaces are more general than \( b \)-metric spaces. Next, we extend Matkowski's fixed-point theorem to Ćirić-type generalized contractions within the framework of \( b \)-suprametric spaces (Theorem \ref{bsupra-gen-matko}). Then, we derive this result as a corollary for strong \( b \)-metric spaces and suprametric spaces (Corollary \ref{strong-bsupra-gen-matko}). \\
 Finally, we investigate a recently introduced type of metric known as interpolative metrics and demonstrate that they are essentially equivalent to \( b \)-metrics (Lemma \ref{not_new}). Our findings extend and generalize several existing results in the literature, as discussed in detail in the subsequent sections.
\section{Main Results}\label{mr}
First, we present  lemma \ref{not_new}, and demonstrate that every interpolative metric is a \(b\)-metric. Then, we provide Example \ref{supexm} demonstrating the generality of \( b \)-suprametric spaces over \( b \)-metric spaces. Next, we extend Matkowski's fixed point theorem to the setting of \( b \)-suprametric spaces. 

\begin{lemma} \label{not_new}
Every \(\alpha\)-\(c\) interpolative metric, where \(0<\alpha<1\) and \(c\ge 0\),  is a \(b\)-metric.
\end{lemma}
\begin{proof}
Starting with the definition of the \(\alpha\)-\(c\) interpolative metric, we have:
\[
\Delta(x, y) \leq \Delta(x, z) + \Delta(z, y) + c \left[ (\Delta(x, z))^\alpha \cdot (\Delta(z, y))^{1-\alpha} \right].
\]
By Young's inequality, for \( \alpha \in (0, 1) \), the following holds:
\[
(\Delta(x, z))^\alpha \cdot (\Delta(z, y))^{1-\alpha} \leq \alpha \Delta(x, z) + (1-\alpha) \Delta(z, y).
\]
Substituting this into the original inequality, we obtain:
\[
\Delta(x, y) \leq \Delta(x, z) + \Delta(z, y) + c \left[ \alpha \Delta(x, z) + (1-\alpha) \Delta(z, y) \right].
\]
Expanding and grouping like terms, we derive:
\[
\Delta(x, y) \leq \left[1 + c \alpha \right] \Delta(x, z) + \left[1 + c (1-\alpha) \right] \Delta(z, y).
\]
Let \( s = \max\{1 + c \alpha, 1 + c (1-\alpha)\} \). Since \( 1 + c \alpha \) and \( 1 + c (1-\alpha) \) are both positive, we conclude:
\[
\Delta(x, y) \leq s \cdot \left[ \Delta(x, z) + \Delta(z, y) \right].
\]
 We see that $\Delta$ is a \(b\)-metric with index \(s\).
\end{proof}
It is worth noting that the following example is inspired by the work of Jleli and Samet \cite{jleli2018new}.
\begin{example}\label{supexm}
Let \(X=\mathcal{R}\) and for all \(x,y\in X\) define:
\begin{equation*}
S(x,y) = \left\{
\begin{array}{rl}
    \exp(x-y)  &   x\not=y \\
    0\quad&\text{otherwise}
\end{array} \right. \qquad
\end{equation*}
\(S\) is not a \(b\)-metric, see Example 2.4. in  \cite{jleli2018new}. For all \(x\not = y,z\in X\):
\begin{align*}
S(x,y)&=\exp(x-y)=\exp(x-z)exp(z-y)\\
&\leq \exp(x-z)+\exp(z-y)+\exp(x-z)exp(z-y)\\
&\leq S(x,z)+S(z,y)+S(x,z)S(z,y).\\
\end{align*}
This implies that \(S\) is a suprametric with \(s=c=1\).
\end{example}
We now present the following result regarding \( b\)-suprametric spaces. As previously mentioned, \( b\)-suprametric spaces generalize \( b \)-metric spaces, offering a broader framework for analysis.
\begin{theorem}\label{bsupra-gen-matko}
Let $(X, \Delta)$ be a complete $b$-suprametric space with parameter $s \geq 1$, and let $\theta \colon [0,\infty) \to [0,\infty)$ satisfy $\theta \in \Theta_1$. For a mapping $T \colon X \to X$, assume:

\begin{enumerate}[label=(\roman*)]
    \item (Contraction Condition) For all $x, y \in X$:
    \begin{equation}\label{fee-matk1}
    \Delta(Tx, Ty) \leq \theta\left(\max\left\{\Delta(x, y), \Delta(x, Tx), \Delta(y, Ty)\right\}\right)
    \end{equation}
    
    \item (Continuity) Either:
    \begin{itemize}
        \item $T$ is continuous, or
        \item $\Delta$ is continuous in each variable
    \end{itemize}
    
    \item  $T$ Has bounded orbits.
\end{enumerate}

Then:
\begin{enumerate}[label=(\alph*)]
    \item For any $x \in X$, the sequence $\{T^n x\}_{n=1}^\infty$ converges to a point $\xi \in X$
    \item $\xi$ is the unique fixed point of $T$ in $X$
\end{enumerate}
\end{theorem}

\begin{proof}
It is well known that \(\Delta(T^{n+1}x,T^{n}x)\to 0\) as \(n\to \infty\) \cite{hussain2018existence, karapinar2025some}. But, for the sake of completeness, we present the proof.  To do this, let \(m\) be any non-negative integer. If \(T^{m+1}x = T^m x\), then \(T^{m}x\) is a fixed point of \(T\), and the sequence \(\{T^n x\}\) converges to \(T^{m}x\). This completes the proof in this case. Now, suppose that \(T^{n+1}x \neq T^n x\) for each non-negative integer \(n\). Let \(n\) be any non-negative integer; we have
\[
\Delta(T^{n+2}x,T^{n+1}x)=\theta(\max\{\Delta(T^{n}x, T^{n+1}x),\Delta( T^{n+1}x,T^{n+2}x)\}).
\]

If the \(\Delta( T^{n+1}x,T^{n+2}x)\) on the right-hand side of above inequality achieves the maximum for some \(n\), then we get:
\[
\Delta(T^{n+2}x,T^{n+1}x)\le \theta(\Delta(T^{n+2}x,T^{n+1}x))<\Delta(T^{n+2}x,T^{n+1}x),
\]
that is a contradiction. Therefore,  from \eqref{fee-matk1}, for each non-negative integer \(n\) we have:
\[
\Delta(T^{n+1}x,T^{n}x)\le \theta(\Delta(T^{n}x,T^{n-1}x)).
\]
Inductively, for each non-negative integer \(n\) we have:
\[
\Delta(T^{n+1}x,T^{n}x)\le \theta^{n}(\Delta(Tx,x)),
\]
Since \(\Delta(Tx,x)\not = 0\), the above inequality implies that 
\begin{equation}\label{asym}
\lim_{n\to\infty}\Delta(T^{n+1}x,T^{n}x)=0.
\end{equation}
 For any non-negative integer \(p\) we have:
\begin{align}\label{indoc}
\Delta(T^{n+p+1}x,T^{n}x)&\le s\{\Delta(T^{n+p+1}x,T^{n+p}x)+\Delta(T^{n+p}x,T^{n}x)\}\nonumber\\
&\quad +c\cdot \Delta(T^{n+p+1}x,T^{n+p}x)\cdot \Delta(T^{n+p}x,T^{n}x).
\end{align}
Using inequalities (\ref{asym}) and (\ref{indoc}), and induction we deduce that:
\begin{equation}\label{induc}
\lim_{n\to\infty}\Delta(T^{n+m}x,T^{n}x)=0\quad \text{ for all non-negative integer } m.
\end{equation}
Suppose, for contradiction, that the sequence $\{T^n x\}_{n\in\mathbb{N}}$ is not Cauchy. Then there exist strictly increasing sequences of positive integers $\{n(k)\}_{k\in\mathbb{N}}$ and $\{m(k)\}_{k\in\mathbb{N}}$ with $n(k), m(k) \to \infty$ as $k \to \infty$, and a constant $L > 0$ such that:
\begin{equation}\label{b0}
    \Delta\left(T^{n(k)}x, T^{m(k)}x\right) \geq L \quad \text{for all } k \in \mathbb{N}.
\end{equation}
Let \(p\) be a non-negative integer. We have:
\begin{align}
\Delta(T^{n(k)}x, T^{m(k)}x) &\le s \big\{ \Delta(T^{n(k)}x, T^{n(k)+p}x) + \Delta(T^{n(k)+p}x, T^{m(k)}x) \big\} \nonumber \\
&\quad+c{ \Delta(T^{n(k)}x, T^{n(k)+p}x)\cdot \Delta(T^{n(k)+p}x, T^{m(k)}x)}\nonumber\\
&\le s \big\{ \Delta(T^{n(k)}x, T^{n(k)+p}x) \nonumber \\
&\quad + s \big\{ \Delta(T^{n(k)+p}x, T^{m(k)+p}x) + \Delta(T^{m(k)+p}x, T^{m(k)}x)\big\}\nonumber\\
& \quad + c\Delta(T^{n(k)+p}x, T^{m(k)+p}x) \cdot \Delta(T^{m(k)+p}x, T^{m(k)}x) \big\} \nonumber \\
&\quad+c\Delta(T^{n(k)}x, T^{n(k)+p}x) \cdot \Delta(T^{n(k)+p}x, T^{m(k)}x).\nonumber
\end{align}
By rearranging the terms on the right-hand side of this inequality, we obtain:
\begin{align}\label{b1}
\Delta(T^{n(k)}x, T^{m(k)}x) &\le s \, \Delta(T^{n(k)}x, T^{n(k)+p}x) + s^{2} \, \Delta(T^{m(k)+p}x, T^{m(k)}x) \nonumber \\
&\quad + s^{2} \, \Delta(T^{n(k)+p}x, T^{m(k)+p}x)\\
&\quad+ sc\Delta(T^{n(k)+p}x, T^{m(k)+p}x) \cdot \Delta(T^{m(k)+p}x, T^{m(k)}x)\nonumber\\
&\quad+c\Delta(T^{n(k)}x, T^{n(k)+p}x) \cdot \Delta(T^{n(k)+p}x, T^{m(k)}x).\nonumber
\end{align}
If \( \liminf_{k \to \infty} \Delta(T^{n(k)+p}x, T^{m(k)+p}x) = 0 \) for some non-negative integer \( p \), then the inequality
\begin{align}\label{b12}
\Delta\left(T^{n(k)+p}x, \, T^{m(k)}x\right) 
&\leq s \Bigl[ \Delta\left(T^{n(k)+p}x, \, T^{m(k)+p}x\right) + \Delta\left(T^{m(k)+p}x, \, T^{m(k)}x\right) \Bigr] \nonumber \\
&\quad + c \cdot \Delta\left(T^{n(k)+p}x, \, T^{m(k)+p}x\right) \Delta\left(T^{m(k)+p}x, \, T^{m(k)}x\right).
\end{align}
 implies that 
\[
\liminf_{k \to \infty} \Delta(T^{n(k)+p}x, T^{m(k)}x) = 0.
\] 
Furthermore, by invoking \eqref{induc}, it can be deduced that the limit inferior of each term on the right-hand side of \eqref{b1} approaches \( 0 \) as \( k \to \infty \). Consequently, this leads to \( L = 0 \), which contradicts our initial assumption. Therefore, we conclude that:
\[
\liminf_{k\to \infty}\Delta(T^{n(k)+p}x, T^{m(k)+p}x)\not=0 \text{ for each non-negative integer } p.
\]
Fix a non-negative integer \(p\). Therefore, we may assume that:
\[
\liminf_{k\to \infty}\Delta(T^{n(k)+j}x, T^{m(k)+j}x)=\delta_{j}>0 \text{ for } j=1,2,\dots,p.
\]
Let \(\delta=\min\{\delta_1,\delta_{2},\dots,\delta_{p}\}\). We have:
\begin{align}\label{b2}
\Delta(T^{n(k)+j}x,T^{m(k)+j}x)&\le \theta\big(\max\big\{\Delta(T^{n(k)+j-1}x,T^{m(k)+j-1}x),\nonumber\\
&\quad \Delta(T^{n(k)+j-1}x,T^{n(k)+j}x), \nonumber\\
&\quad \Delta(T^{m(k)+j-1}x,T^{m(k)+j}x)\big\}\big).
\end{align}
By invoking \eqref{asym}, there exists \(k_0\) such that thefollowing two conditions hold:
\begin{align}\label{b2-3}
\liminf_{k\to \infty}\Delta(T^{n(k)+j}x, T^{m(k)+j}x)>\frac{\delta}{2}\quad\text{for all\quad}k \geq k_0\text{ and } j=1,2,\dots,p.
\end{align}
\begin{align}\label{b3}
\Delta(T^{n(k)+j-1}x,T^{n(k)+j}x) \text { and }& \Delta(T^{m(k)+j-1}x,T^{m(k)+j}x) < \frac{\delta}{2}\nonumber \\
& \text{ for all }k\ge k_0, j\geq 0.
\end{align}
From Inequalities (\ref{b2}),\eqref{b2-3} and (\ref{b3}), we obtain:
\begin{align}
\Delta(T^{n(k)+j}x,T^{m(k)+j}x)&\le \theta\big(\Delta(T^{n(k)+j-1}x,T^{m(k)+j-1}x)\big),\nonumber\\
 &\text{ for all }j=1,2,\dots,p \text{ and } k\ge k_0.
\end{align}
This implies that for all \(k\ge k_0\), we have:
\begin{align}\label{b4}
\Delta(T^{n(k)+p}x,T^{m(k)+p}x)&<\theta\big(\Delta(T^{n(k)+p-1}x,T^{m(k)+p-1}x)\big)\nonumber\\
&<\theta^{2}\big(\Delta(T^{n(k)+p-2}x,T^{m(k)+p-2}x)\big)\nonumber\\
&<\dots\nonumber\\
&<\theta^{p}\big(\Delta(T^{n(k)}x,T^{m(k)}x)\big).
\end{align}
There exists $M>0$ such that:
\begin{equation}\label{b5}
\Delta(T^{n(k)}x,T^{m(k)}x)\leq M \text{ for all } k\in \mathbb{N}.
\end{equation}
Combining (\ref{b4}) and (\ref{b5}), we get:
\begin{equation}
\Delta(T^{n(k)+p}x,T^{m(k)+p}x)\leq \theta^{p}(M) \text{ for all }k\ge k_0.
\end{equation}
From \eqref{b12}
we get:
\begin{align*}
\Delta(T^{n(k)+p}x, T^{m(k)}x)& \leq s \{\theta^{p}(M)+ \Delta(T^{m(k)+p}x, T^{m(k)}x)\} \\
&\quad + c \cdot \Delta\left(T^{n(k)+p}x, \, T^{m(k)+p}x\right) \Delta\left(T^{m(k)+p}x, \, T^{m(k)}x\right).\\
&\quad\text{ for all }k\ge k_0.
\end{align*}
 This implies that the sequence \(\{\Delta(T^{n(k)+p}x, T^{m(k)}x) \}\) is bounded. We observe that, each term on the left-hand side of \eqref{b1}, except for the second term, approaches \(0\) as \(k\to\infty\). Therefore, taking limit inferior  of both sides of \eqref{b1} as \(k\to \infty\), we deduce that:
\begin{align*}
L\leq \liminf_{k\to \infty}\Delta(T^{n(k)}x, T^{m(k)}x) &\le  \liminf_{k\to\infty} s^{2} \Delta(T^{n(k)+p}x, T^{m(k)+p}x)\\
&\leq s^{2} \ \theta^{p}\big(M\big)
\end{align*}
Taking limit as \(p\to \infty\)  implies that \(L=0\),  a contradiction. Therefore, \(\{T^{n}x\}\) is a Cauchy sequence. Since \((X, \Delta)\) is complete, the sequence converges to some point \(z \in X\). Suppose that \(\Delta(z,Tz)>0\), then we have:
\begin{align}
\Delta(T^{n+1}x, Tz) \le\theta(\max\{\Delta(T^{n}x,z),\Delta(T^{n}x,T^{n+1}x),\Delta(z,Tz)\}).
\end{align}
Since the first two sequences on the right-hand side of above inequality tend to \(0\) as \(n\to\infty\), there exists a non-negative integer \(m_0\) such that:
\begin{align}
\Delta(T^{n+1}x, Tz) \le\theta(\Delta(z,Tz)) \text{ for all }n\ge m_0. 
\end{align}
If $\Delta$ is separately continuous, then taking the limit as \( n \to \infty \) yields:
\begin{align}
\Delta(z, Tz) \le\theta(\Delta(z,Tz))<\Delta(z,Tz),
\end{align}
 this contradiction implies that \( \Delta(z, Tz) = 0 \), so \( z \) is a fixed point of \( T \). If \(T\) is continuous, then:
\[
T^{n}x\to z,T^{n+1}x\to Tz \text{ as } n\to \infty.
\]
Since, in a \(b\)-suprametric space the limit is unique, this implies that \(z=Tz\) and the proof is complete.
To prove the uniqueness of the fixed point, suppose \( x \) and \( y \) are two fixed points of \( T \). If \( x \neq y \), then by the contraction condition, we have:
\begin{align}
\Delta(x, y) = \Delta(Tx, Ty) &\leq \theta\left( \max\{\Delta(x, y), \Delta(x, Tx), \Delta(y, Ty)\} \right) \nonumber\\
&= \theta\left(\Delta(x, y)\right) <\Delta(x,y) .\nonumber 
\end{align}
This leads to a contradiction, as \( \Delta(x, y) < \Delta(x, y) \) is impossible. Therefore, \( x = y \), and the fixed point is unique. This completes the proof.
\end{proof}
                    
\begin{cor}\label{strong-bsupra-gen-matko}
Let \((X, \Delta)\) be a complete strong \(b\)-suprametric space or a complete suprametric space  , \(\theta\in \Theta_{1}\), and \(T : X \to X\) be a mapping that  $T$ has bounded orbits, and satisfies the following contraction:
\begin{equation}
\Delta(Tx, Ty) \leq \theta(\max\{\Delta(x, y),\Delta(x, Tx),\Delta( y,Ty)\})\quad \text{for all } x,y\in X.
\end{equation}
Then, \(T\) has a unique fixed point \(z\), and for each \(x \in X\), the sequence \(\{T^n x\}\) converges to \(z\).
\end{cor}
\begin{proof}Every suprametric or strong \(b\)-metric  is jointly continuous. Therefore, the proof follows from Theorem \ref{bsupra-gen-matko}.
\end{proof}
We deduce  Theorem 3.2. \cite{berzig2024nonlinear} and Theorem 2.1 of \cite{berzig2024strong} as a direct corollary of our result as follows.
\begin{cor}(\cite{berzig2024nonlinear}, Theorem 3.2.)
Let \((X, \Delta)\) be a complete \(b\)-suprametric space, \(\theta\in \Theta_{1}\), and \(T : X \to X\) be a mapping that $T$ has bounded orbits, and satisfies the following contraction:
\begin{equation}
\Delta(Tx, Ty) \leq \theta(\Delta(x, y))\quad \text{for all } x,y\in X\nonumber.
\end{equation}
 Then, \(T\) has a unique fixed point \(z\), and for each \(x \in X\), the sequence \(\{T^n x\}\) converges to \(z\).
\end{cor}
\begin{proof}
We observe that \(T\) is continuous, and thus, the proof follows from Theorem \ref{bsupra-gen-matko}.
\end{proof}
Since every \(b\)-metric is a \(b\)-suprametric, we deduce the following {\'C}iri{\'c}-type fixed point theorem in the setting of \(b\)-metric spaces.
\begin{cor}\label{cirictype-bmetric}
Let \((X, \Delta)\) be a complete \(b\)-metric space, \(\theta\in \Theta_{1}\), and \(T : X \to X\) be a mapping that $T$ has bounded orbits, and  satisfies the following contraction:
\begin{equation}
\Delta(Tx, Ty) \leq \theta(\max\{\Delta(x, y),\Delta(x, Tx),\Delta( y,Ty)\})\quad \text{for all } x,y\in X\nonumber.
\end{equation}
 If \(T\) is continuous, or $\Delta$ is separately continuous,  then, \(T\) has a unique fixed point \(z\), and for each \(x \in X\), the sequence \(\{T^n x\}\) converges to \(z\).
\end{cor}

Replacing \(\theta \in \Theta_2\) with \(\theta \in \Theta_1\), and assuming $T$  has bounded orbits, we deduce  Theorem 2.2 in \cite{karapinar2025some} as a direct result of Corollary \ref{cirictype-bmetric} as follows:
\begin{cor}\label{gen-iterpol}
Let \(0<\alpha<1,c\ge0\) and \((X, \Delta)\) be a complete \(\alpha-c\)-interpolative-metric space, \(\theta\in \Theta_{1}\), and \(T : X \to X\) be a mapping $T$ has bounded orbits, and that satisfies the following contraction:
\begin{equation}\label{fee-interpol}
\Delta(Tx, Ty) \leq \theta(\max\{\Delta(x, y),\Delta(x, Tx),\Delta( y,Ty)\})\quad \text{for all } x,y \in X.
\end{equation}
Then, \(T\) has a unique fixed point \(z\) and for each \(x \in X\), the sequence \(\{T^n x\}\) converges \(z\).
\end{cor}
\begin{proof}By Lemma \ref{not_new}, $\Delta$ is a \(b\)-metric. Therefore, by Corollary \ref{cirictype-bmetric}, \(T\) has a unique fixed point and for each \(x \in X\), the sequence \(\{T^n x\}\) converges to the fixed point of \(T\).
\end{proof}
We deduce Theorem 1 in \cite{czerwik1993contraction}, or  Theorem 1 in \cite{kajanto2018note},  as a direct result of Corollary \ref{cirictype-bmetric} as follows.

\section{Discussion}\label{sec12}
We demonstrated that interpolative metric spaces are \( b \)-metric spaces. However, rather than presenting our fixed-point theorem in this setting, we established it within the framework of \(b\)-suprametric spaces, which we proved to be more general than \( b \)-metric spaces.

\section{Conclusion}
 We provided an explicit example to show that $b$-suprametrics are more general than $b$-metrics. Additionally, we unified interpolative metric spaces as a special case of $b$-metrics. In this work, we provided a Matkowski-type fixed point theorem to generalized contractions of \'{C}iri\'{c} type in $b$-suprametric spaces, demonstrating their broader applicability compared to $b$-metric spaces. Our findings generalize several theorems in the literature and offer new insights into fixed point theory in generalized metric spaces, paving the way for future research.
\backmatter

\section*{Declarations}
We affirm that we have no competing interests.







\bibliography{sn-bibliography}

\end{document}